\documentclass[10pt,reqno]{amsart}
\usepackage[sort]{cite}
\usepackage{amsfonts} 
\usepackage{amsmath} 
\usepackage{amssymb} 
\usepackage{amsthm} 
\usepackage{latexsym} 
\usepackage{mathrsfs} 
\usepackage{mathtools} 
\usepackage{slashed}  

\usepackage{verbatim}
\usepackage{cases}
\usepackage[dvips]{epsfig}
\usepackage{epsf} 
\usepackage[bookmarksnumbered,pdfpagelabels=true,plainpages=false,colorlinks=true,
            linkcolor=black,citecolor=black,urlcolor=black]{hyperref}
\usepackage{url}
\usepackage{color}
\usepackage{xcolor}
\usepackage{graphicx}
\usepackage{enumitem} 

\usepackage{pifont} 
\usepackage{upgreek} 
\usepackage{fancyhdr} 
\usepackage{calligra} 
\usepackage{marvosym} 
\usepackage[percent]{overpic} 
\usepackage{pict2e} 
\usepackage[hypcap=false]{caption} 
\usepackage{wasysym} 
\usepackage{accents}
\theoremstyle{plain}
\newtheorem{theorem}{Theorem}[section]

\newtheorem{proposition}[theorem]{Proposition}
\newtheorem{corollary}[theorem]{Corollary}

\theoremstyle{definition}
\newtheorem{remark}[theorem]{Remark}

\newtheorem{definition}[theorem]{Definition}

\numberwithin{equation}{section}
\allowdisplaybreaks

\newcommand{\cA}{\mathcal{A}}
\newcommand{\cB}{\mathcal{B}}

\newcommand{\cD}{\mathcal{D}}
\newcommand{\cE}{\mathcal{E}}

\newcommand{\cG}{\mathcal{G}}

\newcommand{\cL}{\mathcal{L}}

\newcommand{\cS}{\mathcal{S}}

\newcommand{\bC}{\mathbb{C}}
\newcommand{\bE}{\mathbb{E}}

\newcommand{\bM}{\mathbb{M}}
\newcommand{\bR}{\mathbb{R}}
\newcommand{\bS}{\mathbb{S}}

\newcommand{\sE}{\mathsf{E}}
\newcommand{\sM}{\mathsf{M}}
\newcommand{\sT}{\mathsf{T}}

\begin{document}
\title[Magnetoviscoelastic fluids]{Well-posedness for magnetoviscoelastic fluids in 3D}

\author{Hengrong Du}
\address{Department of Mathematics\\
        Vanderbilt University\\
        Nashville, Tennessee\\
        USA}
\email{hengrong.du@vanderbilt.edu}

\author{Yuanzhen Shao}
\address{The University of Alabama\\ 
	Tuscaloosa, Alabama \\
	USA}
\email{yshao8@ua.edu}

\author{Gieri Simonett}
\address{Department of Mathematics\\
        Vanderbilt University\\
        Nashville, Tennessee\\
        USA}
\email{gieri.simonett@vanderbilt.edu}

\thanks{This work was supported by a grant from the Simons Foundation (\#426729 and \#853237, Gieri Simonett).}

\subjclass[2020]{Primary: 35Q35, 35Q74, 35K59, 35B40. Secondary: 76D03, 76A10.}



\keywords{Landau-Lifshitz-Gilbert system,
quasilinear parabolic equation, strong well-posedness, normally stable, Lyapunov function, convergence to equilibria.}

\begin{abstract}
We show that the system of equations describing a magnetoviscoelastic fluid
in three dimensions can be cast as a quasilinear parabolic system.
 Using the theory of maximal $L_p$-regularity, we 
establish existence and uniqueness of local strong solutions
and we show that each solution is smooth (in fact analytic) in space and time. 
Moreover, we give a complete characterization of the set of equilibria and show 
that solutions that start out close to a   constant  equilibrium exist globally and converge to a (possibly different) 
  constant  equilibrium.
Finally, we show that every solution that is eventually bounded in the topology of the state space 
exists globally and converges to the   set of equilibria. 
\end{abstract}

\maketitle


\section{Introduction}\label{S:Intro}

We will study the following system of equations that models the evolution of a magnetoviscoelastic fluid
\begin{equation}
\label{magneto sys}
\left\{\begin{aligned}
\partial_t u   + u   \cdot \nabla u  -\mu_s \Delta u   +\nabla \pi   &=-\nabla \cdot(\nabla m \odot \nabla m)  + \nabla \cdot (F F^{\sT}) 
 &&\text{in} &&   \bR_+\times\Omega, \\
\nabla \cdot u   &=0 &&\text{in}&& \bR_+\times\Omega ,\\
u &=0 &&\text{on}&&\bR_+\times\partial\Omega  , \\
\partial_t F  + u  \cdot\nabla F  -(\nabla u )^{\sT} F  &=\kappa \Delta F  &&\text{in}&&\bR_+\times\Omega,\\
F  &=0   &&\text{on}&&\bR_+\times\partial\Omega, \\
\partial_t m   + u  \cdot\nabla m   &= - \alpha m \times (m \times \Delta m)  - \beta m  \times \Delta m   &&\text{in}&&\bR_+\times\Omega,\\
\partial_\nu m   &=0   &&\text{on}&&\bR_+\times\partial\Omega, \\
|m |  &=1   &&\text{in}&& \bR_+\times\Omega, \\
 (u(0), F(0), m(0))& =(u_0, F_0, m_0) &&\text{in}&& \Omega. \\
\end{aligned}\right.
\end{equation}
Here, $\Omega \subset \bR^3$ is a bounded connected  $C^3$-domain   with outward unit normal field $\nu$.  
The unknowns $(u,F,m): \bR_+\times \Omega\to \bR^3\times \bM^3\times\bR^3$
denote the fluid velocity, the deformation tensor field and the magnetization field, respectively,
while  $\pi$ is the pressure.
Moreover,   $\bM^3$   stands for the set of all $(3\times 3)$-real matrices.
The parameters $\alpha, \beta>0$ are the so-called Gilbert damping and  the exchange constant, while
$\mu_s$ and $\kappa$ are the dynamic viscosity and dissipative coefficient, respectively.

\noindent
The notation $\nabla m \odot \nabla m$ means $\nabla m (\nabla m)^{\sT}$.
Hence,
$\nabla m \odot \nabla m$ 
is a symmetric tensor with coefficients $[\nabla m \odot \nabla m]_{ij}= \partial_i m \cdot \partial_j m$. 
\goodbreak
 \eqref{magneto sys} is a coupled system of equations containing
 \begin{itemize}
 \item the incompressible Navier--Stokes equations for the velocity field $u$  
 and including in addition magnetic and elastic terms in the stress tensor,
 \item a transport-stretch-dissipative system for the deformation tensor $F$,
 \item a convected Landau--Lifshitz--Gilbert system for the magnetization field $m$.
 \end{itemize}
 As a multi-physical hydrodynamics model, \eqref{magneto sys} enjoys the following energy dissipation property:
%
\begin{equation*}
 \label{energy-dissipation}
  \begin{aligned}
    &\quad \frac{d}{dt}\int_{\Omega}\frac{1}{2}\left( |u|^2+| F|^2+|\nabla  m|^2 \right)\, dx  \\
    &\hspace{3cm}=-\int_\Omega\left( \mu_s  |\nabla u|^2 + \kappa | \nabla F |^2  + \alpha |\,  \Delta m + |\nabla m|^2 m\,  |^2\right)\,dx, \\
  \end{aligned} 
\end{equation*}
see Proposition~\ref{pro: energy-dissipation}.

The system \eqref{magneto sys} was first introduced in \cite{BFLS18, For16l}.
This model can describe  the motion of   fluids with micromagnetic and elastic particles such as ferrofluids \cite{BAB99, AR09} and magnetorheological fluids \cite{Wer14}.
Existence of  weak solutions in 2D was established in \cite{BFLS18}
under a smallness condition on the initial data by using a Galerkin approximation.
In  \cite{KKS21}, the authors extended these results
under more general assumptions on the elastic energy density.
Moreover, they proved local-in-time existence of strong solutions and they established a weak-strong uniqueness property.
Recently, the authors in \cite{KorSch22} obtained global weak solutions to \eqref{magneto sys} with partial regularity in a 
2D periodic domain by a careful blow-up analysis near singularities.

The main difficulty in constructing global weak solutions to \eqref{magneto sys} is caused by the lack of sufficient integrability in the a priori energy estimates for the stress tensor term $\nabla  m \odot\nabla  m $.

In 3D, the authors in  \cite{For16l,GKMS21,SchZab18} consider the
simplified system
  \begin{equation}
 \label{eqn:GLappro}
  \left\{
    \begin{aligned}
      \partial_t u    +u    \cdot \nabla u   -\mu_s \Delta u   +\nabla \pi 
      &=-\nabla\cdot \left( \nabla m \odot\nabla m  \right) +\nabla\cdot (F F^{\sT}) , \\  
      \nabla\cdot  u  &=0,  \\
      \partial_t F  + u  \cdot \nabla F  -(\nabla u )^{\sT}  F   &=\kappa\Delta F ,  \\
      \partial_t m  + u  \cdot \nabla m  
      &=\Delta m  -\frac{1}{\varepsilon^2}(| m  |^2-1) m ,
  \end{aligned}
  \right. 
 \end{equation}
where 
the constraint $|m|\equiv 1$ is replaced
by the Ginzburg--Landau penalization term $\frac{1}{\varepsilon^2}(| m |^2-1)^2$. 

In \cite{For16l}, the author adapted the approach in \cite{LinLiu95} to show the existence of weak solutions to \eqref{eqn:GLappro} with the combination of a Galerkin approximation scheme and a fixed point argument. The weak-strong uniqueness of solutions to \eqref{eqn:GLappro} was established in \cite{SchZab18} under the Prodi--Serrin condition. In case the initial values have higher regularity, the authors in \cite{GKMS21} obtained the well-posedness of strong solutions to \eqref{eqn:GLappro} via a priori estimates that are uniform in the approximate solutions.

We would like to point out that 
 the regularization term $\kappa \Delta F$ in \eqref{magneto sys} and \eqref{eqn:GLappro} 
 with $0<\kappa \ll 1$  plays an important role in the mathematical analysis. 
If $\kappa=0$, the evolution of the deformation tensor field becomes hyperbolic, and in this case, even in 2D, the existence of weak solutions to incompressible viscoelastic fluids $(m=0)$ with large initial data remains an open problem. Local well-posedness of strong solutions to \eqref{eqn:GLappro} without the regularization term was established in \cite{Wenjing2018Local} in a periodic domain in two or three dimensions.

 From the viewpoint of modeling, $F=0$ represents a liquid phase that contains no elastic solid particles. 
 We refer the reader to \cite{LiuWal01} for more details.

\smallskip
To the best of our knowledge, there are so far no existence results for system \eqref{magneto sys} in 3D.
In our approach, we consider  \eqref{magneto sys} as a quasilinear system
and prove that the system is parabolic.
We can then apply the theory of maximal $L_p$- regularity to establish short time existence and uniqueness of strong solutions, 
see Theorem~\ref{Thm: existence and uniqueness}.
In Sections 3 and 4, we show that the set of equilibria of 
\eqref{magneto sys} is given by 
\begin{equation*}
\cE=\{(0,0_3,m_*,\pi_*)\},
\end{equation*}
where $m_*\in H^2_q(\Omega,\bR^3)$ solves the nonlinear constrained elliptic problem
\begin{equation}
\label{equilibrium-problem}
\left\{\begin{aligned}
\Delta m_* + |\nabla m_*|^2 m_* &= 0 &&\text{in}&&\Omega,\\
  |m_*| &\equiv 1  &&\text{in}&&\Omega,\\
\partial_\nu m_*  &=0   &&\text{on}&&\partial\Omega 
\end{aligned}\right.
\end{equation}
and $\pi_*=-\frac{1}{2}|\nabla m_*|^2 +C$
for some constant $C$. 

In particular, we have that 
\begin{equation*}
\cE_{c}:=\{(0,0_3,m_*,\pi_*)\in \bR^3\times\bM^3\times \bS^2\times\bR\}\subset\cE.
\end{equation*}
We call $\cE_{c}$ the set of constant equilibria.
We can prove that all constant equilibria are normally stable, and that each solution that starts
out close to a constant equilibrium exists globally and converges to a (possibly different) constant equilibrium.
Moreover, we show that any solution that is bounded in an appropriate topology exists globally 
and converges to the set $\cE$ of equilibria.

In case we choose $(u_0,F_0)=(0,0)$, system \eqref{magneto sys} 
reduces to the well-known Landau-Lifshitz-Gilbert equation
\begin{equation}
\label{LLG}
\left\{\begin{aligned}
\partial_t m    &= - \alpha m \times (m \times \Delta m)  - \beta m \times \Delta m  &&\text{in}&&  \bR_+\times\Omega,\\
\partial_\nu m &=0   &&\text{on}&&\bR_+\times\partial\Omega, \\
|m |  &=1   &&\text{in}&& \bR_+\times\Omega, \\
m(0)& =m_0 &&\text{in}&& \Omega. \\
\end{aligned}\right.
\end{equation}
In this case, we obtain the energy dissipation relation
\begin{equation*}
 \frac{d}{dt}\int_{\Omega}\frac{1}{2}|\nabla  m|^2 \, dx = - \int_\Omega \alpha |\,  \Delta m + |\nabla m|^2m\,  |^2 \,dx. 
\end{equation*}
By the same arguments as in Section 4, we can conclude that the set of equilibria of \eqref{LLG} is given by 
the solutions of \eqref{equilibrium-problem}.
Hence, all the results established for system \eqref{magneto sys} remain true for the Landau-Lifshitz-Gilbert equation.
A similar result was obtained in \cite{Melcher2012} in case  $\Omega=\mathbb{R}^n$ with $n\ge 3$. 

\smallskip
Finally, we mention that all of our results remain valid in 2D, that is, in case
$\Omega\subset \bR^2$ and $(u,F,m):\Omega\to \bR^2\times \bM^2\times \bR^3$.

\medskip\noindent
\textbf{Notation:} 
For the readers' convenience, we list here some notation and conventions used throughout the manuscript.

In the following, all vectors $a=(a_1,\cdots, a_n)\in \bR^n$ are viewed as column vectors. 
For two vectors $a,b\in\bR^n$, the Euclidean inner product is denoted by $a\cdot b$.
Given two matrices $A,B\in \bM^n$, the Frobenius matrix inner product $A:B$ is given by
$$
A:B={\rm Tr} (AB^{\sT}),
$$
where ${}^{\sT}$ is the transpose.
Suppose $\Omega$ is an open subset of $\bR^n$.
If $u\in C^1(\Omega;\bR^n)$, we set $\nabla u(x)= e_j\otimes \partial_ju(x)$
for $x\in\Omega$.
Hence, for $u=(u_1,\cdots , u_n)\in C^1(\Omega; \bR^n)$, we have 
\begin{equation}
\label{gradientvector}
[\nabla u(x)]_{ij}= \partial_i u_j(x), \;\; 1\le i, j\le n, \;\; x\in \Omega.
\end{equation}
We note that $[\nabla u(x)]^{\sT}$ corresponds to the Fr\'echet derivative of $u$ at $x\in\Omega$.

If $A\in C^1(\Omega;\bM^n)$, its divergence $\nabla \cdot A$ is the vector function defined by
\begin{equation}
\label{divergence-matrix}
(\nabla \cdot A)(x)=(\partial_j A(x))^{\sT}e_j, \;\; x\in \Omega.
\end{equation}
Hence, if $A=[a_{ij}]\in C^1(\Omega;\bM^n)$, its divergence is given by
 $$[(\nabla\cdot A)(x)]_i=\partial_j a_{ji}(x), \;\; i=1,\cdots , n,\;\; x\in \Omega. $$
Here and in the sequel, we use the summation convention, indicating that terms with repeated indices are added.
We note   that  \eqref{gradientvector} and
 \eqref{divergence-matrix} imply
$$
\nabla\cdot (\nabla u)=\Delta u , \quad u\in C^2(\Omega;\bR^n),
$$
and 
\begin{equation}
\label{divergence-property}
(\nabla\cdot A)\cdot u = \nabla\cdot (Au)- A:\nabla u, \quad A\in C^1(\Omega;\bM^n), \ u\in C^1(\Omega;\bR^n).
\end{equation}
For a matrix $A\in C^1(\Omega;\bM^n)$, we set
$|\nabla A|^2=\partial_j A:  \partial_j A.$

\smallskip
For functions 
$f,g\in L_2(\Omega; \bR^m)$, 
$$(f|g)_\Omega =\int_\Omega f\cdot g\, dx$$ denotes the $L_2$-inner product.
For any  Banach space $X$, $s\ge 0$, $p,q\in (1,\infty)$, 
$B^s_{pq}(\Omega;X)$ denote the  $X-$valued Besov spaces,
whereas $H^s_{q}(\Omega;X)$ are the  Bessel-potential spaces. 
When the choice of $X$ is clear from the context, we will just write $B^s_{pq}(\Omega)$ or $H^s_{q}(\Omega)$.

\goodbreak
\medskip
For $p\in (1,\infty)$ and $\mu\in (0,1]$, the $X$-valued $L_p$-spaces with temporal weight are defined by
$$
L_{p,\mu}((0,T);X):=\left\{ f: (0,T)\to X: \, t^{1-\mu}f(t)\in L_p((0,T);X)  \right\}.
$$
Similarly,
$$
H^1_{p,\mu}((0,T);X):=\left\{ f \in  L_{p,\mu}((0,T);X)\cap H^1_1((0,T);X):\,  f'(t)\in L_{p,\mu}((0,T);X) \right\}.
$$

For any two Banach spaces $X$ and $Y$, the notation $\cL(X,Y)$ stands for the set of all bounded linear maps from $X$ to $Y$ and $\cL(X):=\cL(X,X)$.
 
\section{Existence and uniqueness of solutions }\label{Section:existence and unique}


In this section, we show how to formulate system~\eqref{magneto sys} as a quasilinear equation.
Using the theory of maximal $L_p$-regularity, we establish existence and uniqueness of local in time solutions,
and we show that solutions have additional time regularity.
We start by expressing the term 
$$ \alpha m\times (m \times \Delta m) + \beta m \times \Delta m$$
in a form that is more convenient for our analysis. 
By the well-known identity $a\times (b\times c)= (a\cdot c)b -(a\cdot b)c$, we have
\begin{equation*}
\label{double cross product 1}
m\times (m\times \Delta m)= (m \cdot \Delta m) m - |m|^2 \Delta m .
\end{equation*}
By using the facts that $|m|=1$ and
$
0= \Delta |m|^2 =2|\nabla m|^2 + 2m \cdot \Delta m,
$
we obtain
\begin{equation}
\label{double cross product 2}
m\times (m \times \Delta m)= - (\Delta m + |\nabla m|^2 m ),
\end{equation}
 provided $m$ is sufficiently smooth.
 Setting
\begin{align*}
\sM(m)=
\begin{bmatrix}
0 &  -m_3 &m_2 \\
m_3 & 0 &  -m_1 \\
-m_2 & m_1 & 0
\end{bmatrix}, \quad m=(m_1,m_2, m_3),
\end{align*}
we can write $m\times \Delta m = \sM(m) \Delta m.$
Hence, under the constraint $|m|\equiv 1$, \eqref {magneto sys} is equivalent to  the following system 
\begin{equation}
\label{magneto sys 2}
\left\{\begin{aligned}
\partial_t u   + u   \cdot \nabla u   -\mu_s   \Delta u   +\nabla \pi  
&=   -\nabla \cdot (\nabla m \odot \nabla m)  + \nabla \cdot (F F^{\sT}  ) 
 &&\text{in}&& \bR_+\times\Omega,\\
\nabla \cdot u  &=0 &&\text{in}&&\bR_+\times\Omega ,\\
u &=0 &&\text{on}&&\bR_+\times\partial \Omega  , \\
\partial_t F  -\kappa \Delta F &=   (\nabla u )^{\sT} F    -  u  \cdot\nabla F  &&\text{in}&&\bR_+\times\Omega,\\
F &=0   &&\text{on}&&\bR_+\times\partial\Omega, \\
\partial_t m + u  \cdot\nabla m  &=   ( \alpha I_3- \beta \sM(m) )\Delta m   + \alpha |\nabla m |^2 m  &&\text{in}&&\bR_+\times\Omega,\\
\partial_\nu m   &=0   &&\text{on}&&\bR_+\times\partial\Omega, \\
|m |  &=1   &&\text{in}&& \bR_+\times\Omega, \\
(u(0), F(0), m(0))& =(u_0, F_0, m_0) &&\text{in}&& \Omega, 
\end{aligned}\right.
\end{equation}
where $I_3$ is the $3\times 3$ identity matrix.
Neglecting the constraint $|m|\equiv 1$,
we have
\begin{equation}
\label{magneto sys 3}
\left\{\begin{aligned}
\partial_t u   + u  \cdot \nabla u -\mu_s   \Delta u   +\nabla \pi  
&=   -\nabla \cdot (\nabla m \odot \nabla m)   + \nabla \cdot (F F^{\sT}  ) 
 &&\text{in}&& \bR_+\times\Omega ,\\
\nabla \cdot u &=0 &&\text{in}&&\bR_+\times\Omega ,\\
u  &=0 &&\text{on}&&\bR_+\times\partial \Omega  , \\
\partial_t F -\kappa \Delta F &=   (\nabla u  )^{\sT} F   -  u  \cdot\nabla F   &&\text{in}&&\bR_+\times\Omega,\\
F  &=0   &&\text{on}&&\bR_+\times\partial\Omega, \\
\partial_t m  + u \cdot\nabla m &=   ( \alpha I_3- \beta \sM(m) )\Delta m   + \alpha |\nabla m |^2 m &&\text{in}&&\bR_+\times\Omega,\\
\partial_\nu m   &=0   &&\text{on}&&\bR_+\times\partial\Omega, \\
(u(0), F(0), m(0))& =(u_0, F_0, m_0) &&\text{in}&& \Omega. 
\end{aligned}\right.
\end{equation}
We will first study the unconstrained system~\eqref{magneto sys 3}, 
and then show in a second step that the constraint $|m|\equiv 1$ is preserved 
in case $|m_0|\equiv 1$.

\medskip
The main tool to study \eqref{magneto sys 3} is the  theory of maximal $L_p$-regularity.
For $\theta\in (0,\pi]$, the open sector with angle $2\theta$ is denoted by
$$\Sigma_\theta:= \{\omega\in \mathbb{C}\setminus \{0\}: |\arg \omega|<\theta \}. $$
\begin{definition}
Let $X$ be a complex Banach space, and $\cA$ be a densely defined closed linear operator in $X$ with dense range. $\cA$ is called sectorial if $\Sigma_\theta \subset \rho(-\cA)$ for some $\theta>0$ and
$$ \sup\{\|\lambda (\lambda +\cA)^{-1}\|_{\cL(X)} : \lambda\in \Sigma_\theta \}<\infty. $$
The class of sectorial operators in $X$ is denoted by $\cS(X)$. 
\end{definition}

To introduce the notion of maximal  $L_p$-regularity, let us consider the following abstract Cauchy problem on $[0,T]$
\begin{equation}
\label{S4: Cauchy problem}
\left\{\begin{aligned}
\partial_t u(t) +\cA u(t) &=f(t), &&t\in (0,T),\\
u(0)&=0 . &&
\end{aligned}\right. 
\end{equation}
\begin{definition}
Assume that $X_1\overset{d}{\hookrightarrow}X_0$ is some densely embedded Banach couple.
Suppose that $\cA\in \cS(X_0)$ with $D(\cA)=X_1$.
The operator $\cA$ is said to have  the property of maximal $L_p$-regularity if for any fixed $T>0$ and
$$f\in L_p((0,T); X_0) ,$$
\eqref{S4: Cauchy problem} has a unique solution
$$u\in L_p((0,T); X_1) \cap H^1_p((0,T); X_0)  .$$
We denote the set of all operators $\cA\in S(X)$ which enjoy the property of maximal $L_p$-regularity by 
$$\cA\in \mathcal{MR}_p(X_1, X_0).$$
\end{definition}
\noindent
We refer to \cite{PruSim16} for additional background information.

\medskip
Let $P_H: L_q(\Omega;\bR^3) \to L_{q,\sigma}(\Omega;\bR^3) $   be the Helmholtz projection, where 
$$
L_{q,\sigma}(\Omega;\bR^3) :=P_H (L_q(\Omega;\bR^3) )
$$
is the space of all solenoidal vector fields in $L_q(\Omega; \bR^3)$.
Setting 
$$H^2_{q,\sigma}(\Omega;\bR^3) := H^2_q(\Omega;\bR^3)\cap L_{q,\sigma}(\Omega;\bR^3),$$
we let  $\cA_q: D(\cA_q)\to L_{q,\sigma}(\Omega;\bR^3)$ be the Stokes operator, defined by
\begin{equation*}
\label{Stokes}
\cA_q u:=-\mu_s   P_H \Delta u,\quad D(\cA_q):= \{u\in H^2_{q,\sigma}(\Omega;\bR^3): u=0 \text{ on }\partial\Omega \}.
\end{equation*}
Similarly, we can define 
$\cG_q: D(\cG_q)\to L_q(\Omega;\bM^3)$  by 
\begin{equation*}
\label{Gq}
{\cG_q} F:=-\kappa \Delta F,\quad 
 D(\cG_q):= \{F\in H^2_q(\Omega;\bM^3): F=0 \text{ on }\partial\Omega \}.
\end{equation*}
Further, given any $m\in C(\overline{\Omega};\bR^3)$,  the operator $\cD_q(m): D(\cD_q(m))\to  L_q(\Omega;\bR^3) $ is defined by
\begin{equation*}
\label{Dq}
\begin{aligned}
&\cD_q (m) h:= - ( \alpha I_3 - \beta \sM(m) )\Delta h,  \\
 &D(\cD_q(m)):= \{h\in H^2_q(\Omega;\bR^3): \partial_\nu h=0 \text{ on }\partial\Omega \}.
\end{aligned}
\end{equation*}
Next, we set 
\begin{equation}
\label{def-B}
[\cB_q(m)h]_i=  \partial_i m \cdot \Delta h+   \nabla m : \partial_i \nabla  h ,   \quad i=1,2,3.
\end{equation}
It follows that 
\begin{equation}
\label{B(m)m}
\cB_q(m)m=\nabla\cdot (\nabla m \odot \nabla m) \quad\text{for each $m\in H^2_q(\Omega; \bR^3)$.}
\end{equation}
Note that $\cB_q(m) \in \cL(D(\cD_q(m)), L_q(\Omega;\bR^3))$ for any $m\in C^1(\overline{\Omega};\bR^3)$.
Finally, we define the spaces
$$
X_0=L_{q,\sigma}(\Omega;\bR^3) \times L_q(\Omega;\bM^3) \times L_q(\Omega;\bR^3)
$$
and
$$
X_1= D(\cA_q) \times D(\cG_q) \times D(\cD_q(m)).
$$
It is well known that $\cA_q$ and $\cG_q$ enjoy the property of maximal $L_p$-regularity, cf. \cite{GHHSS10, Giga, Sol77}
and \cite[Section 6.3, Chapter 7]{PruSim16}.
To deal with $\cD_q(m)$ for $m\in C(\overline{\Omega};\bR^3)$, we set 
$$\cD_q(m(x),\xi):=(\alpha I_3- \beta \sM(m(x)) )|\xi|^2,\quad x\in\Omega, \; \; \xi \in \bR^3,$$ 
for the symbol of the differential operator $\cD_q(m)$.
An easy computation shows that 
\begin{equation*}
\sigma (\cD_q(m(x),\xi))= \{\alpha, \alpha \pm  i\beta |m(x)|\},\quad x\in\Omega, \;\; \xi \in \mathbb S^2,
\end{equation*}
where $\sigma $ denotes the spectrum.
Since $m\in C(\overline{\Omega};\bR^3)$, $\cD_q(m(x),\xi)$ is  normally elliptic for every $x\in \overline{\Omega}$, see for instance \cite[Definition~6.1.1]{PruSim16}.
By \cite[Theorem~6.3.2]{PruSim16}, $\cD_q(m)$ has the property of maximal $L_p$-regularity.
Then,  the operator $A_q:X_1\to X_0$ defined by
\begin{equation}
\label{def-Aq}
A_q(m)=
\begin{bmatrix}
\cA_q &  0 & P_H \cB_q(m) \\
0 & \cG_q & 0 \\
0 & 0 & \cD_q(m)
\end{bmatrix}
\end{equation}
enjoys the property of maximal $L_p$-regularity for every $m\in C^1(\overline{\Omega};\bR^3) $ as well, due to its upper triangular structure.

Indeed, given any $f=(f_1,f_2,f_3)\in L_p((0,T);X_0)$, we consider  the system
\begin{equation}
\label{Cauchy problem 2}
\left\{\begin{aligned}
\partial_t z +A_q (m)z  &=f(t), &&t\in (0,T),\\
z(0)&=0 , &&
\end{aligned}\right. 
\end{equation}
where $z=(v,G,h)$. By the maximal $L_p$-regularity property of $\cG_q$ and $\cD_q(m)$, one can  find 
for each $m\in C(\overline{\Omega};\bR^3)$ a (unique) solution
$$(g,h)\in L_p((0,T); H^2_q(\Omega ;\bM^3\times \bR^3))\cap H^1_p((0,T);L_q(\Omega ;\bM^3\times\bR^3))$$
for the system
\begin{equation*}
\label{Cauchy problem 3}
\left\{\begin{aligned}
\partial_t G +\cG_q G&=  f_2  &&\text{in}&&\Omega,\\
G &=0   &&\text{on}&&\partial\Omega, \\
\partial_t h + \cD_q(m)h &= f_3 &&\text{in}&&\Omega   ,\\
\partial_\nu h  &=0   &&\text{on}&&\partial\Omega, \\
 (G(0),h(0))&=(0,0)  . &&
\end{aligned}\right. 
\end{equation*}
Easy computations show that $\cB_q(m)h \in L_p((0,T); L_q(\Omega; \bR^3))$ for $m\in C^1(\overline{\Omega};\bR^3) $.
From the maximal $L_p$-regularity property of $\cA_q$, we thus infer that there exists a (unique) vector $v\in L_p((0,T); H^2_{q,\sigma}(\Omega; \bR^3))\cap H^1_p((0,T);L_{q,\sigma}(\Omega ; \bR^3))$ that solves
\begin{equation*}
\label{Cauchy problem 4}
\left\{\begin{aligned}
\partial_t v +\cA_q v &= - P_H \cB_q(m)h +  f_1  &&\text{in}&&\Omega,\\
v &=0   &&\text{on}&&\partial\Omega, \\
 v(0)&=0 . &&
\end{aligned}\right. 
\end{equation*}
Hence $(v,G,h)$ is the unique solution of \eqref{Cauchy problem 2}.
This shows that 
\begin{equation}
\label{max-reg-Aq}
A_q(m)\in \mathcal{MR}_p(X_1,X_0) \quad \text{for each $m\in C^1(\overline{\Omega};\bR^3)$}.
\end{equation}

In addition, we define for $z=(u,F,m)$
\begin{equation}
\label{def-G}
G(z): =\left(P_H \left[\nabla \cdot (F F^{\sT}) -u\cdot \nabla u \right]   ,    (\nabla u)^{\sT} F   -u \cdot\nabla F , \alpha|\nabla m|^2 m -u \cdot\nabla m  \right).
\end{equation}

Given any $1<p,q<\infty$, $T>0$ and $\mu\in (1/p,1]$, we set
$$
\bE_{0,\mu}(T):=L_{p,\mu}((0,T); X_0) \quad\text{and} \quad 
\bE_{1,\mu}(T):=L_{p,\mu}((0,T); X_1)\cap H^1_{p,\mu}((0,T); X_0).
$$
It is well known that 
$$
\bE_{1,\mu}(T) \hookrightarrow C([0,T]; X_{\gamma,\mu}) \quad\text{where}\quad X_{\gamma,\mu}:=(X_0,X_1)_{\mu-1/p,p}.
$$
See \cite{PruSim04}, or \cite[Theorem 3.4.8]{PruSim16}. 
Observe that by \cite[Theorem 3.4]{Ama00}  and \cite[Theorem~4.3.3]{Trib78}, the triple $(u,F,m)\in B^{2\mu-2/p}_{qp }(\Omega;\bR^{15})$
belongs to $X_{\gamma,\mu}$ iff
\begin{equation}
\label{interpolation}
\begin{aligned}
&u\in B^{2\mu-2/p}_{qp,\sigma}(\Omega;\bR^3) &&\text{and} \quad  u=0 \text{ on } \partial \Omega ,\\
&F\in B^{2\mu-2/p}_{qp }(\Omega;\bM^3)  &&\text{and} \quad F=0 \text{ on } \partial \Omega, \\
&m\in  B^{2\mu-2/p}_{qp }(\Omega;\bR^3) && \text{and} \quad \partial_\nu m=0 \text{ on } \partial \Omega,
\end{aligned}
\end{equation}
where $B^{2\mu-2/p}_{qp,\sigma}(\Omega;\bR^3):=B^{2\mu-2/p}_{qp }(\Omega;\bR^3)\cap L_{q,\sigma}(\Omega;\bR^3)$.
In order for $\partial_\nu m$ to be defined, we assume that $2\mu -2/p -1/q>1$. 

One readily verifies that
\begin{equation}
\label{analytic}
A_q\in C^\omega(X_{\gamma,\mu},\cL(X_1,X_0)),\quad 
G\in C^\omega(X_{\gamma,\mu},X_0),
\end{equation}
with $\omega$ being the notation for real analyticity,
as long as 
$$
X_{\gamma,\mu}\hookrightarrow C^1(\overline{\Omega};\bR^{15}).
$$  
The above embedding holds whenever  $\mu \in \left(\frac{1}{2}+\frac{1}{p}+\frac{3}{2q},1 \right]$.

By  the definitions \eqref{def-B}, \eqref{def-Aq} and \eqref{def-G}, and the relation \eqref{B(m)m},
one sees that system \eqref{magneto sys 3} can be recast as the abstract evolutionary system  

\begin{equation}
\label{abstract evolution}
\partial_t z + A_q(m) z=G(z),\quad z(0)=z_0=(u_0,F_0,m_0).
\end{equation}
We have the following result on existence and uniqueness of solutions of \eqref{abstract evolution}.
\goodbreak
\begin{proposition}
\label{pro: abstract}
Suppose   $\mu \in \left(\frac{1}{2}+\frac{1}{p}+\frac{3}{2q},1 \right]$ and let $z_0\in X_{\gamma,\mu}$. 
Then there exists $T=T(z_0)$ such that the evolution equation \eqref{abstract evolution}
admits a unique solution $z\in \bE_{1,\mu}(T)$.
Each solution can be extended to a maximal existence interval $[0, T_+(z_0))$ in the sense that  
\begin{enumerate}[label={\em(\roman*)}]
\item either $T_+(z_0)=\infty$ or
\item $\lim\limits_{t\to T_+(z_0)} z(t)  $ does not exist in $X_{\gamma,\mu}$.
\end{enumerate}
Moreover, $z$  enjoys the additional regularity properties
\begin{equation}
\label{additiional-regularity-z}
z\in C([0,T_+); X_{\gamma,\mu})\cap C^\omega ((0,T_+); X_1) \cap  C^\omega ((0,T_+)\times \Omega ; \bR^{15}).  
\end{equation}
\end{proposition}
\begin{proof}
The existence, uniqueness and time regularity follow from \eqref{max-reg-Aq}, \eqref{analytic}
and \cite[Theorems 5.1.1 and 5.2.1 and Corollary~5.1.2]{PruSim16}, see also
\cite[Theorem~2.1]{KohPruWil10}.
The joint space-time regularity~\eqref{additiional-regularity-z} can be proved by means of the parameter trick in \cite{EschPruSim03}, see also \cite[Section~9.4.1]{PruSim16}.
\end{proof}
Next, we show that solutions of \eqref{abstract evolution} give rise to solutions of \eqref{magneto sys 3}, and vice versa.
\begin{proposition}
\label{pro: equivalence}
Let $T>0$ be given. The following statements are equivalent:
\begin{itemize}
\item[{\rm (a)}]  \eqref{abstract evolution}  has a solution $(u,F,m)\in \bE_{1,\mu}(T) $.
\vspace{1mm}
\item[{\rm (b)}] \eqref{magneto sys 3} has a solution $(u,F,m,\pi)\in \bE_{1,\mu}(T)\times L_{p,\mu}((0,T); \dot{H}^1_q(\Omega))$.  
\end{itemize}
\end{proposition}
\begin{proof}
(a)$\Rightarrow$(b): 
Suppose $z= (u,F,m)\in \bE_{1,\mu}(T) $  solves \eqref{abstract evolution} on  $[0,T]$.
Let
\begin{equation*}
\label{Def of v}
v=   \mu_s   \Delta u  - u \cdot \nabla u -  \nabla\cdot(\nabla m\odot \nabla m) + \nabla \cdot (F F^{\sT} ).  
\end{equation*}
Then $v\in L_{p,\mu}((0,T); L_q(\Omega;\bR^3))$.
For $t\in (0,T)$, let  $\nabla \psi_{v(t)} \in L_q(\Omega;\bR^3)$ be the unique solution of
$$
(\nabla \psi_{v(t)} | \nabla \phi )_\Omega =(v(t) | \nabla \phi )_\Omega ,\quad \forall \phi\in \dot{H}^1_{q'}(\Omega),
$$ 
where $(\cdot| \cdot)_\Omega $ is the inner product of $L_2(\Omega;\bR^3)$  and $q'$ is the H\"older dual of $q$. 
Then $P_H v(t) =v(t)-\nabla \psi_{v(t)}$ by the definition of the Helmholtz projection.
Let $\pi=\psi_v$.  Then  $\pi\in L_{p,\mu}((0,T); \dot{H}^1_q(\Omega)),$ 
and noting that $P_H\partial_t u= \partial_t u$, we conclude that
 $(u,F,m,\pi)$ is a solution of \eqref{magneto sys 3} in the regularity class $ \bE_{1,\mu}(T)\times L_{p,\mu}((0,T); \dot{H}^1_q(\Omega))$.
 \medskip\\
 \noindent
(b)$\Rightarrow$(a): 
Suppose $(u,F,m,\pi)\in \bE_{1,\mu}(T)\times L_{p,\mu}((0,T); \dot{H}^1_q(\Omega))$ solves \eqref{magneto sys 3}.
Applying $P_H$ to the equation  governing $u$  in \eqref{magneto sys 3}, it is an easy task to check that
$(u,F,m)\in \bE_{1,\mu}(T) $ solves \eqref{abstract evolution}.
\end{proof}
We are now ready for our main result on existence and uniqueness of solutions for 
system \eqref{magneto sys 2}, or equivalently,  system \eqref{magneto sys}.
\goodbreak
\begin{theorem}
\label{Thm: existence and uniqueness}
Let  $p,q\in (1,\infty)$ and  $\mu \in \left(\frac{1}{2}+\frac{1}{p}+\frac{3}{2q},1 \right]$.
Suppose that
$$
z_0= (u_0,F_0,m_0)\in B^{2\mu-2/p}_{qp,\sigma}(\Omega;\bR^3)\times B^{2\mu-2/p}_{qp }(\Omega;\bM^3) \times B^{2\mu-2/p}_{qp}(\Omega;\bR^3)
$$
satisfies the compatibility conditions
$
(u_0, F_0, \partial_\nu m_0)=0 \text{  on } \partial\Omega.
$
Then there exists a unique solution
\begin{equation*}
(u ,F ,m ,\pi)\in \left[ H^1_{p,\mu}((0,T); X_0) \times L_{p,\mu}((0,T); X_1) \right]\times L_{p,\mu}((0,T); \dot{H}^1_q(\Omega))
\end{equation*}
of \eqref{magneto sys 3} for some $T=T(z_0)>0$. 
Each solution can be extended to a maximal existence interval $[0, T_+(z_0))$.
Moreover, $(z,\pi)=(u,F,m,\pi)$ enjoys the additional regularity
\begin{equation}
\label{time-regularity}
z\in C([0,T_+); X_{\gamma,\mu})\cap C^\omega ((0,T_+); X_1) \text{ and }   (z,\pi)\in   C^\omega ((0,T_+)\times \Omega ; \bR^{16}).
\end{equation}
If $|m_0|\equiv 1$, then the solution also satisfies
\begin{equation}
 |m(t)|\equiv 1,\quad t\in [0, T_+(z_0)).
 \end{equation}
\end{theorem}

\begin{proof}
The assertions in the first part of the statement follow readily from Propositions~\ref{pro: abstract} and \ref{pro: equivalence}.
It then only remains to show that the condition $|m(t)|\equiv 1$ holds for every $t\in [0, T_+(z_0))$, provided 
$|m_0|\equiv 1$.

Suppose then that $|m_0|\equiv 1$.
Let $T\in (0,T_+(z_0))$ be fixed and set $\varphi=|m|^2-1$. 
We  note that
\begin{equation*}
\begin{aligned}
&m\in C^1((0,T); H^2_q(\Omega;\bR^3)) \\
&\varphi\in C([0,T]; B^{2\mu-2/p}_{qp}(\Omega ))\cap C^1 ((0,T); H^2_q(\Omega )). \\
\end{aligned}
\end{equation*}
Indeed, \eqref{time-regularity} implies that $m\in C([0,T]; B^{2\mu-2/p}_{qp}(\Omega;\bR^3 )) \cap C^1((0,T); H^2_q(\Omega;\bR^3))$. The condition $\mu \in \left(\frac{1}{2}+\frac{1}{p}+\frac{3}{2q},1 \right]$ guarantees that $B^{2\mu-2/p}_{qp}(\Omega )$ and $H^2_q(\Omega)$ are Banach algebras. The asserted regularity of $\varphi$ thus holds.
Taking the dot product of the equation
$$
\partial_t m + u \cdot\nabla m = \alpha(\Delta m +|\nabla m|^2 m)-\beta m \times \Delta m
$$
with $m$ and using the relations 
$
\partial_t |m|^2= 2 \partial_t m \cdot m, \ \Delta |m|^2= 2\Delta m\cdot m + 2|\nabla m|^2 
$
results in
\begin{equation}
\label{eqn:phiequ}
\left\{\begin{aligned}
\partial_t \varphi  + u\cdot \nabla \varphi -   \alpha \Delta \varphi -   2\alpha  |\nabla m|^2 \varphi &= 0 &&\text{in}&&\Omega,\\
\partial_\nu \varphi  &=0   &&\text{on}&&\partial\Omega, \\
  \varphi (0) &=0. &&
\end{aligned}\right.
\end{equation}
Multiplying both sides of \eqref{eqn:phiequ} with $\varphi$ and integrating over $\Omega$ yields
\begin{equation*}
  \frac{d}{dt}\int_{\Omega} \frac{1}{2}\varphi^2\, dx+   \alpha \int_{\Omega}|\nabla \varphi|^2 \, dx
  =2\alpha\int_{\Omega}|\nabla m|^2 \varphi^2 \, dx,\quad t\in (0,T).
\end{equation*}
As $m\in C([0,T]; B^{2\mu-2p}_{qp}(\Omega;\bR^3))\hookrightarrow C([0,T]; C^1(\overline\Omega;\bR^3))$, we 
obtain the following integral inequality
\begin{equation*}
  \frac{d}{dt}\int_\Omega \varphi^2(t)\, dx\le 4\alpha\|\nabla m(t)\|^2_\infty \int_{\Omega}\varphi^2(t)\, dx,\quad t\in (0,T).
\end{equation*}
Applying the Gronwall inequality we get
\begin{equation*}
  \max_{0\le t\le T} \int_{\Omega}\varphi^2 (t)\, dx\le \exp\left( 4\alpha\int_{0}^{T} \|\nabla m(t)\|^2_\infty \, dx \right)\int_{\Omega} \varphi^2(0) \, dx=0. 
\end{equation*}
This implies $\varphi\equiv 0$ in $Q_T=[0,T] \times \Omega$. In other words, $|m|\equiv 1$ in $Q_T$.  
As this is true for every $T\in (0, T_+(z_0))$ we obtain that $|m(t)|\equiv 1$ for  any $t\in (0,T_+(z_0))$.
As \eqref{magneto sys 2} and \eqref{magneto sys} are equivalent, we have proved the assertions of the theorem.
\end{proof}

\begin{remark}
Let $p,q,\mu$ be as in Theorem~\ref{Thm: existence and uniqueness}.
The assertions of Theorem~\ref{Thm: existence and uniqueness},   with exception of the higher time regularity stated
in \eqref{time-regularity}, 
 still hold if we pose the nonhomogeneous boundary conditions
\begin{equation}
\label{boundary data}
(u, F, m)=(u_D, F_D, m_N)  \quad \text{on }\partial\Omega,
\end{equation}
where 
\begin{equation*}
\begin{aligned}
&u_D\in F^{1-1/2q}_{pq,\mu}((0,T);L_q(\partial\Omega;\bR^3)) \cap L_{p,\mu}((0,T);W^{2-1/q}_{q }(\partial\Omega;\bR^3)) \\
&
F_D\in F^{1-1/2q}_{pq,\mu}((0,T);L_q(\partial\Omega;\bM^3)) \cap L_{p,\mu}((0,T);W^{2-1/q}_{q }(\partial\Omega;\bM^3)) \\
&
m_N\in F^{1/2-1/2q}_{pq,\mu}((0,T);L_q(\partial\Omega;\bR^3)) \cap L_{p,\mu}((0,T);W^{1-1/q}_{q }(\partial\Omega;\bR^3))
\end{aligned}
\end{equation*}
and the initial data satisfy the compatibility conditions
$$
(u_D(0), F_D(0), m_N(0))=(u_0,F_0, \partial_\nu m_0) \quad \text{on }\partial\Omega.
$$
See \cite[Theorem~6.3.2]{PruSim16}.
Here $F^s_{pq,\mu}$ are the Triebel-Lizorkin spaces with temporal weight.
\end{remark}

\section{Stability and asymptotic behavior}\label{Section:Stability and asymptotic behavior}

The last two sections are devoted to a discussion of the asymptotic behavior of solutions $(u,F,m,\pi)$ to \eqref{magneto sys}. 
In view of Proposition~\ref{pro: equivalence}, the pressure $\pi$ can be  obtained from $z=(u,F,m)$. 
For this reason, it suffices to restrict our attention to a solution $z=(u,F,m)$ of \eqref{abstract evolution}.

The $3$-dimensional subspace 
$$\cE_0:=\{0\}\times \{0_3\}\times\bR^3\;\;\text{of $X_1$}$$
is clearly contained in the set $\cE_1$ of equilibria of \eqref{abstract evolution},
where $0_3$ is the $3\times3$ matrix with zero entries.
 We refer to Remark~\ref{rem:4.3}(a) for more information on $\cE_1$.

At each $z_*=(0,0_3, m_*)\in 	\cE_0$, the linearization of \eqref{abstract evolution} is given by 
\begin{equation}
\label{Linearization}
\partial_t z + A_*  z =0  ,\quad z(0)=z_0,
\end{equation}
where $z=(u,F,m)$ and 
\begin{align*}
A_* z= (-\mu_s P_H \Delta u  ,  -\kappa \Delta F ,  (\beta\sM(m_*)-\alpha I_3) \Delta m  ).
\end{align*}
Since $\Omega$ is bounded, the spectrum of $A_*$   consists only of eigenvalues. 
Suppose that $A_*z=\lambda z$ for some $\lambda \in \bC$.
By elliptic regularity theory, we can assume that $z\in H^2_q(\Omega;\bC^{15})$
for $q\ge 2$.
Taking the inner product of $A_*z=\lambda z$  with $\overline z$, where $\overline z$ denotes the complex conjugate of $z$,
direct computations lead to 
\begin{align*}
{\rm Re}\,\lambda \left(\|u\|_2^2 + \|F\|_2^2 + \|m\|_2^2  \right)= \mu_s \|\nabla u\|_2^2 + \kappa \|\nabla F\|_2^2 +\alpha \|\nabla m\|_2^2,
\end{align*} 
which implies that ${\rm Re}\,\lambda \geq 0$. 
Here we have used the anti-symmetry of $\sM(m_*)$ to conclude that
$$
{\rm Re}\,( \sM(m_*) \Delta m | \overline{m})_\Omega=0.
$$
Indeed, as the entries of $\sM(m_*)$ are constant, we obtain 
$$
(\sM(m_*) \Delta m | \overline m)_\Omega =(\Delta (\sM(m_*) m) | \overline m)_\Omega=(\sM(m_*)m | \overline{\Delta m})_\Omega,
$$
where we set $z\cdot \overline{w}=z_j \overline{w_j}$ for $z,w\in\bC^3$.
The  anti-symmetry of $\sM(m_*)$ implies
$$
(\sM(m_*) \Delta m | \overline m)_\Omega = -(\Delta m | \sM(m_*)\overline{m})_\Omega =-\overline{(\sM(m_*)m | \overline{\Delta m}}_\Omega.
$$
This readily yields ${\rm Re}\, (\sM(m_*) \Delta m | \overline m)_\Omega=0$.
When ${\rm Re}\,\lambda = 0$, one concludes from the above that
$$
\|\nabla u\|_2=\|\nabla F\|_2=\|\nabla m\|_2=0.
$$
Combined with the boundary conditions, this shows that $z=(u,F,m)\in \cE_0 $.
Further, we infer that $A_* z=0$. Thus $\sigma(A_*)\cap i\bR=\{0\} $ and $N(A_*)=\cE_0$.

To show $\{0\}$ is a semi-simple eigenvalue, we will prove that $N(A_*)=N(A_*^2)$. Assume that $w=(v,f,h)\in N(A_*^2)$. Then there exists $z=(0,0_3, m)\in  N(A_*)$ such that $A_* w=z$. Then by the divergence theorem,   the boundary condition $\partial_\nu h=0$  and the fact that  $m$ as well as  $m_*$ are constant, 
\begin{align*}
\|z\|_2^2 = (A_* w | z) =  ( (\beta\sM(m_*) - \alpha I_3)\Delta h | m)  =0.
\end{align*}
We conclude that $z=0$ and thus $w\in N(A_*)$. This shows that $\{0\}$ is semi-simple. 
As $\cE_0$ is a linear space, we clearly have $T_{z_*} \cE_0 =N(A_*).$

If follows from \cite[Remark 2.2]{PrSiZa09}, see also \cite[Remarks 5.3.2]{PruSim16}, that all equilibria close to 
$z_*$ are contained in a manifold ${\mathcal M}$ of dimension $3= {\rm dim}(N(A_*))$,
where we used the fact that the center space $X^c$ coincides with $N(A_*)$ as $\{0\}$ is semi-simple. 
Since the dimension of $\cE_0$ is also 3, we conclude that there exists an open neighborhood $V_*\subset X_1$ of $z_*$ such that
${\mathcal M}\cap V_*=\cE_0\cap V_*$. Hence, the neighborhood $V_*$ contains no other equilibria than the elements of $\cE_0$,
that is, $V_*\cap \cE_0=V_*\cap \cE_1$. 

\medskip
We have, thus, shown that  $z_*$ is normally stable, see  \cite[Theorem~5.3.1]{PruSim16} for a definition.

\begin{theorem}
\label{Thm: Stability}
Let  $p,q\in (1,\infty)$ and   $\mu \in \left(\frac{1}{2}+\frac{1}{p}+\frac{3}{2q},1 \right]$.

Then each equilibrium  $z_*=(0,0_3, m_*) $ with $m_*\in \bS^2$ is stable in the topology of $X_{\gamma,\mu}$.
There  exists $\varepsilon>0$ such that any solution $ (u, F,m,\pi)$ of \eqref{magneto sys} 
with initial value  $z_0=(u_0,F_0,m_0)\in X_{\gamma,\mu}$ satisfying $\|z_0-z_*\|_{X_{\gamma,\mu}}\leq \varepsilon$  exists globally
and converges to some $z_\infty = (0,0_3,m_\infty) $ with $m_\infty\in \bS^2$ in 
the topology of  $X_{\gamma,1}$ 
at an exponential rate as $t\to \infty$.
\end{theorem}
\begin{proof}
Given an equilibrium $z_*=(0,0_3, m_*)\in 	\cE_0$, we infer from  \cite[Theorem~5.3.1]{PruSim16}
 and  \cite[Proposition~5.1]{MaPrSi19} 
 that there exists $\varepsilon>0$ such that any solution $z=(u, F,m )$ of \eqref{abstract evolution} with initial data 
$z_0=(u_0,F_0,m_0)\in   X_{\gamma, \mu}  $ satisfying  the conditions $|m_0|\equiv 1$ and $\|z_0-z_*\|_{X_{\gamma,\mu}}\leq \varepsilon$ exists globally  and converges 
at an exponential rate to some $z_\infty = (0,0_3,m_\infty ) $ with $m_\infty= $ constant, in the topology of
 $X_{\gamma, 1}$ as $t\to \infty$.
By Proposition~\ref{pro: equivalence}, 
we can determine a pressure $\pi$ such that $(z,\pi)$ solves 
 \eqref{magneto sys 3} on $\bR_+$. 
Furthermore, since $|m_0|\equiv 1$, we infer that $|m(t)|\equiv 1$ for all $t\ge 0$, 
which implies that $(z,\pi) $ solves \eqref{magneto sys} on $\bR_+$.  Finally, we conclude that $m_\infty\in \bS^2$.
\end{proof}
\goodbreak
 \goodbreak
\section{Lyapunov functional and global solutions}\label{Section: global existence}
Let 
\begin{equation}
\label{Lyapunov functional}
\sE:=\sE(u,F,m):=\frac{1}{2}\int_\Omega \left( |u|^2 + |F|^2 + |\nabla m|^2 \right)\, dx.
\end{equation}
We show that the energy $\sE$ is dissipated.
\begin{proposition}
\label{pro: energy-dissipation}
Let $(u,F,m,\pi)$ be a solution of \eqref{magneto sys} with initial value $z_0$
satisfying the assertions of Theorem~\ref{Thm: existence and uniqueness}.
Then
\begin{equation*}
 \label{energy-dissipation-2}
  \begin{aligned}
    \frac{d}{dt}\sE (t)
    &=-\int_\Omega\left( \mu_s  |\nabla u(t,x)|^2 + \kappa | \nabla F(t,x) |^2  + \alpha |  \Delta m(t,x) + |\nabla m(t,x)|^2m(t,x)\,  |^2\right)\,dx \\
 \end{aligned}
\end{equation*}
for $t\in (0,T_+(z_0))$.
Moreover, $\sE$ is a strict Lyapunov functional for \eqref{magneto sys}.
\end{proposition}
\begin{proof}
Let  $z_0$ be an initial value  satisfying the assumptions of Theorem~\ref{Thm: existence and uniqueness}.
Then \eqref{magneto sys}  admits a unique solution $(u,F,m,\pi)$ in the regularity class  stated in the Theorem.
In particular, $z=(u,F,m)$ enjoys the regularity property
$$z\in C([0,T_+); X_{\gamma,\mu})\cap C^1 ((0,T_+);X_1),$$
with $T_+=T_+(z_0)$.
In the following, we suppress the time variable $t\in (0, T_+)$.
A straightforward computation, using the boundary condition $\partial_\nu m=0$, yields 
\begin{equation*}
\frac{d}{dt} \sE
=    \int_\Omega \left( \partial_t u \cdot u  + \partial_t F : F -\partial_t m \cdot \Delta m \right)\, dx. \\
\end{equation*}
We have
\begin{equation}
\label{u-derivative}
\begin{aligned}
  &\int_\Omega  \partial_t u \cdot u \,dx \\
 &\hspace{.5cm} =\int_\Omega \left[ \mu_s \Delta u - u \cdot \nabla u -\nabla \pi
-  \nabla\cdot (\nabla m\odot \nabla m) + \nabla \cdot (F F^{\sT} ) \right]\cdot u \, dx  \\
&\hspace{.5cm}=\int_\Omega  \left( - \mu_s |\nabla u|^2 - (u\cdot \nabla m)\cdot \Delta m  - FF^{\sT}: \nabla u\right)\, dx,
\end{aligned}
\end{equation}
where we used $\nabla\cdot u=0$, the boundary condition $u=0$ on $\partial\Omega$,  \eqref{divergence-property}, 
and the relations
\begin{equation*}
\label{odot term}
\nabla \cdot(\nabla m \odot \nabla m)=    \nabla m\,  \Delta m +\frac{1}{2} \nabla (|\nabla m|^2), 
\quad (\nabla m\,\Delta m)\cdot u = (u\cdot \nabla m)\cdot \Delta m.
\end{equation*}
Moreover, 
\begin{equation}
\label{F-derivative}
\begin{aligned}
 \int_\Omega  \partial_t F : F \, dx 
&= \int_\Omega \left[   (\nabla u)^{\sT}  F -  u \cdot\nabla F+\kappa \Delta F \right] : F \, dx  \\
&=\int_\Omega \left(FF^{\sT}: \nabla u -\kappa |\nabla F|^2 \right)\,dx,
\end{aligned}
\end{equation}
where we employed the condition $\nabla \cdot u=0$, the boundary condition $F=0$ on $\partial\Omega$, 
and the  relations 
$$  (\nabla u)^{\sT} F :F  = F  F^{\sT}  : \nabla u ,\quad  2 (u \cdot\nabla F) : F =u \cdot\nabla |F|^2.
$$
Observing that
$
(\partial_t m + u\cdot \nabla m)\cdot m=0,
$
we obtain 
\begin{equation}
\label{m-derivative}
\begin{aligned}
& \int_\Omega (\partial_t m  + u\cdot \nabla m )\cdot \Delta m\, dx  \\
&\quad =\int_\Omega (\partial_t m  + u\cdot \nabla m )\cdot  (\Delta m +| \nabla m|^2 m)\, dx   \\
&\quad  =\int_\Omega (\alpha(\Delta m +|\nabla m|^2 m) - \beta\, m \times \Delta m)\cdot  (\Delta m +| \nabla m|^2 m)\, dx \\
&\quad = \alpha \int_\Omega | \Delta m +|\nabla m|^2m |^2\, dx \\
\end{aligned}
\end{equation}
as $m\times \Delta m$ is perpendicular to both $m$ and $\Delta m$.
Combining the results in \eqref{u-derivative}--\eqref{m-derivative}
readily yields the assertion.

\medskip
Hence, $\sE$ is non-increasing along solutions and, thus, is a Lyapunov functional. 
If, for any time $t\in I:= (t_1,t_2)\subset (0, T_+(z_0))$ with some $0\leq t_1<t_2$,  $\frac{d}{dt} \sE(t)=0$,  then 
$$\|\nabla u(t)\|_2=\|\nabla F(t)\|_2=\| \Delta m(t) + |\nabla m(t)|^2 m(t)\|_2=0.$$ 
Combining with the boundary conditions, we infer that
\begin{equation*}
u(t)=0 , \quad F(t)=0_3, \quad t\in I.
\end{equation*}
This readily yields $(\partial_t u(t), \partial_t F(t)) = (0,0_3)$ for all $t\in I$.
Moreover, the condition $\|\Delta m(t) + |\nabla m(t)|^2 m(t)\|_2=0 $  implies that
\begin{equation}
\label{parallel}
\Delta m(t) + |\nabla m(t)|^2 m(t)  = 0 \quad \text{in } \Omega ,\quad t\in I.
\end{equation}
Taking the cross product of both sides of \eqref{parallel} by $m(t)$ yields
$$
m(t) \times \Delta m(t)= -|\nabla m(t)|^2 m(t) \times m(t)=0  \quad \text{in } \Omega.
$$
Therefore,
$ 
\partial_t m(t) = 0 
$ 
for all $t\in I$. 
Hence, $(\partial_t u(t),\partial_t F(t), \partial_t m(t))=(0,0_3,0)$ for all $t\in I$, and this means that
the system is at equilibrium for $t\in I$.
To sum up, we have proved that
$\sE: X_{\gamma,\mu}\to \bR$ is a strict Lyapunov functional for \eqref{magneto sys}.
\end{proof}
\noindent
The arguments above additionally yield a characterization of the set of equilibria of \eqref{magneto sys}.
\begin{corollary}
\label{cor:equilibria}
The set of equilibria of \eqref{magneto sys} is given by 
\begin{equation}
\label{equilibria}
\cE=\{(0,0_3,m_*,\pi_*)\},
\end{equation}
where $m_* \in H^2_q(\Omega) $ solves the constrained nonlinear elliptic problem
\begin{equation}
\label{equilibirum-equation-2}
\left\{\begin{aligned}
\Delta m_* + |\nabla m_*|^2 m_* &= 0 &&\text{in}&&\Omega,\\
  |m_*| &\equiv 1  &&\text{in}&&\Omega,\\
\partial_\nu m_*  &=0   &&\text{on}&&\partial\Omega ,
\end{aligned}\right.
\end{equation}
and $\pi_*=-\frac{1}{2}|\nabla m_*|^2 +C$
for some constant $C$.
\end{corollary}
\begin{proof}
We have already shown in the proof of Proposition~\eqref{pro: energy-dissipation} 
that any equilibrium of \eqref{magneto sys} is given by $(0,0_3,m_*)$, where
$m_*$ solves~\eqref{equilibirum-equation-2}.
Hence, at equilibrium, we are left with the relation
$\nabla \pi_* = -\nabla \cdot (\nabla m_* \otimes \nabla m_*).$
We have
\begin{equation*}
\begin{aligned}
[-\nabla \cdot (\nabla m_* \otimes \nabla m_*)]_i &=- \partial_i  m_*\cdot \Delta m_* - (\partial_i \partial_j m_*)\cdot \partial_j m_* \\
&=  \partial_i  m_*\cdot  m_* |\nabla m_*|^2 - \frac{1}{2} \partial_i |\nabla m_*|^2  
= - \frac{1}{2} \partial_i |\nabla m_*|^2,
\end{aligned}
\end{equation*}
where we used the relations $\Delta m_* = -|\nabla m_*|^2 m_*$ and  $\partial_i m_*\cdot m_*=\frac{1}{2} \partial_i |m_*|^2=0$.
Hence, $\nabla \pi_*=-\frac{1}{2} \nabla |\nabla m_*|^2$ and the assertion for $\pi_*$ follows.
\end{proof}

\begin{remark}
\label{rem:4.3}
(a) We note that for solutions of system~\eqref{magneto sys 3}, that is, 
in case the condition $|m|\equiv 1$ is dropped in~\eqref{magneto sys 2}, we can only conclude that
\begin{equation*}
 \label{energy-dissipation-3}
  \begin{aligned}
    \quad \frac{d}{dt}\sE
   = -  \int_\Omega  \left( \mu_s  |\nabla u|^2 + \kappa | \nabla F |^2  + \alpha (  |\Delta m|^2 + |\nabla m|^2(m\cdot \Delta m))\right)\,dx.
  \end{aligned} 
\end{equation*}
As the term $m\cdot \Delta m$ does not have a sign,
we can no longer derive the characterization \eqref{equilibria} for the set  of equilibria, $\cE_1$, of ~\eqref{magneto sys 3}, 
respectively \eqref{abstract evolution}.
However, as shown in Section~\ref{Section:Stability and asymptotic behavior}, we can conclude that for every $z_*\in\cE_0$ there exists a neighborhood $V_*$ in $X_1$ such that
$\cE_1\cap V_* = \cE_0\cap V_*$.
\medskip\\
(b) 
It is claimed in \cite[Lemma 5.2]{HNPS14}, see also \cite[Lemma 12.2.4]{PruSim16},
that the nonlinear problem~\eqref{equilibirum-equation-2}
admits only constant solutions $m_*\in\bS^2$.
However, this assertion is not correct in the form stated, as the following example shows:
Let $\Omega=\{x\in\bR^3: 0<r_1<|x|<r_2\}$ and $m_*:\Omega \to\bS^2$ be defined by $m_*(x)=x/|x|.$
Then $m_*$ is a (non-constant) solution of ~\eqref{equilibirum-equation-2}.
\end{remark}

\goodbreak
\begin{theorem}\label{Thm: global existence}
Let $p,q,\mu$, $z_0$ and $ T_+(z_0)$  be as in Theorem~\ref{Thm: existence and uniqueness}.
Suppose that the solution $(u,F,m,\pi)$ of \eqref{magneto sys} satisfies
$$
z=(u,F,m)\in BC([\delta, T_+(z_0)); X_{\gamma,\bar \mu})
$$
for some $\delta\in (0, T_+(z_0))$ and   $\bar \mu \in (\mu, 1].$ 
Then $z$ exists globally and 
$ 
{\rm dist}(u(t),\cE)\to 0 
$ 
in $X_{\gamma,1}$ as $t\to \infty$, where $\cE$ is the set of equilibria of \eqref{magneto sys}.
\end{theorem}
\begin{proof}
Given any initial value $z_0$, we define the $\omega$-limit set of \eqref{abstract evolution}  as
$$
\omega(z_0):=\{w\in X_{\gamma,\mu}: \exists t_n\to \infty \text{ s.t. } \| z(t_n)-w\|_{X_{\gamma,1}}=0 \text{ as }n\to \infty\}.
$$
\cite[Theorem 5.7.1]{PruSim16} implies that $z(\cdot) $ exists globally and the orbit $\{z(t)\}_{t\geq \delta}$ is relatively compact in $X_{\gamma,1}$.
By \cite[Theorem~5.7.2]{PruSim16}, $\omega(z_0)$ is nonempty, compact and $\omega(z_0)\subset \cE$.
Further,   we can infer that $ 
{\rm dist}(z(t),\cE)\to 0 
$ 
in $X_{\gamma,1}$ as $t\to \infty$.
\end{proof}



\section*{Acknowledgements}
The first author would like to thank Professor Changyou Wang for suggesting this problem and many helpful discussions. 

 \end{document}